\documentclass[reqno]{amsart}
\usepackage{amssymb,amsmath,amsthm}
\allowdisplaybreaks
\theoremstyle{plain}
\newtheorem{lemma}{Lemma}
\newtheorem{theorem}{Theorem}
\newtheorem{proposition}{Proposition}
\newtheorem{corollary}{Corollary}
\newtheorem{conjecture}{Conjecture}
\newtheorem{problem}{Problem}

\theoremstyle{definition}
\newtheorem{definition}{Definition}

\theoremstyle{remark}
\newtheorem{remark}{Remark}

\numberwithin{equation}{section}

\newcommand{\qpoch}[2]{(#1;#2)_\infty}
\newcommand{\qpochn}[3]{(#1;#2)_{#3}}
\newcommand{\trinum}[1]{t_{#1}}
\newcommand{\fr}{\frac}
\newmuskip\pFqskip
\mathchardef\pFcomma=\mathcode`, 
\newcommand*\pFq[5]{%
  \begingroup
  \begingroup\lccode`~=`,
  \lowercase{\endgroup\def~}{\pFcomma\mkern\pFqskip}%
 \mathcode`,=\string"8000
 {}_{#1}\phi_{#2}\biggl[\genfrac..{0pt}{}{#3}{#4};#5\biggr]%
 \endgroup
 }
 \newmuskip\pGqskip
\mathchardef\pGcomma=\mathcode`, 

\title[Two-color series]{On a two-color partition series and its companions}

\author[G. E. Andrews and M. El Bachraoui]{George E. Andrews and Mohamed El Bachraoui}
\address{The Pennsylvania State University, University Park, Pennsylvania 16802}
\email{andrews@math.psu.edu}
\address{Department of Mathematical Sciences, United Arab Emirates University, PO Box 15551, Al-Ain, UAE}
\email{melbachraoui@uaeu.ac.ae}


\begin{document}

\begin{abstract}
We study the two-color distinct-part series \(S_1(q)\), equivalently Andrews' generating function \(v_d(q)\) for strictly concave compositions, and its odd and even companions \(T_o(q)\) and \(T_e(q)\). We determine the coefficients of \(S_1(q)\) modulo \(4\) and obtain a complete criterion for the resulting Ramanujan-type progressions. For the even companion, we give a direct overpartition interpretation of its coefficients and show that two natural partition families are each counted by half of those coefficients. For the eta-normalized odd companion \(C(q)=(q;q)_\infty T_o(q)\), we prove a quintic self-similarity, derive exact vanishing relations and infinite sign changes for its coefficients, and show that \(c(n)\) can be nonzero only when \(24n+28\) is represented by \(x^2+3y^2\).
\end{abstract}

\date{\textit{\today}}
\thanks{First author partially supported by Simons Foundation Grant 633284}
\maketitle

\section{Introduction and main results}\label{sec:intro}

Throughout we use the standard $q$-shifted factorial notation; see, for example~\cite[Chapter~1]{GasperRahman}. Thus
\[
(a;q)_0=1,\qquad
\qpochn{a}{q}{n}=\prod_{j=0}^{n-1}(1-aq^j),\qquad
\qpoch{a}{q}=\prod_{j=0}^{\infty}(1-aq^j),
\]
and
\[
(a_1,\dots,a_r;q)_n=(a_1;q)_n\cdots(a_r;q)_n,
\qquad
(a_1,\dots,a_r;q)_\infty=(a_1;q)_\infty\cdots(a_r;q)_\infty.
\]
Integer partitions in which each part may occur in two colors (red and blue) have been studied extensively in recent years. For references on two-color partitions closely related to the present work, see, for instance,~\cite{Andrews1987, Andrews2021, AndrewsElBachraouiPositive, AndrewsElBachraouiCongruences, AndrewsKumar, BanerjeeBringmannKeith, ChernWang}.

In this paper, the main series is
\[
S_1(q):=\sum_{n\ge 0} q^n\,\qpoch{-q^{n+1}}{q}^{2}.
\]
Thus $S_1(q)$ is the generating function for the number of two-color partitions into distinct
parts whose smallest part occurs in one prescribed color only, while every larger part may occur in either color or in both colors.
The constant term counts the empty partition.
The same series also appears in Andrews' work on concave and convex compositions~\cite{AndrewsConcaveConvex}.
Indeed, if $V_d(n)$ denotes the number of strictly concave compositions of $n$ in the notation of~\cite{AndrewsConcaveConvex}, and
\[
v_d(q):=\sum_{n\ge 0}V_d(n)q^n,
\]
then
\[
v_d(q)=S_1(q).
\]
Here the central part contributes $q^n$, and the two strict sides independently select distinct parts greater than $n$.
Thus the congruence results for $s_1(n)$ below may also be read as congruence results for strictly concave compositions.
We record a simple companion identity.
\begin{remark}\label{rem:both-smallest}
If both colors are required to occur in the smallest part, then the natural generating function is
\[
S_2(q):=\sum_{n\ge 1} q^{2n}\,\qpoch{-q^{n+1}}{q}^{2}.
\]
A short calculation gives
\[
S_2(q)=3(-q;q)_\infty^2-1-2S_1(q),
\]
so this variant produces no essentially new series.
\end{remark}
Writing
\[
S_1(q) = \sum_{n\geq 0} s_1(n) q^n,
\]
our first main result is the following complete description modulo $4$ for the sequence $\{s_1(n)\}$.
\begin{theorem}\label{thm:mod4}
For every $n\ge 0$,
\[
s_1(n)\equiv
\begin{cases}
(-1)^r \pmod{4}, & \text{if } n=\dfrac{r(r+1)}{2} \text{ for some } r\ge 0, \\
0 \pmod{4}, & \text{otherwise.}
\end{cases}
\]
In particular, $s_1(n)$ is odd if and only if $n$ is triangular.
\end{theorem}

\begin{remark}\label{rem:parity-from-andrews}
The parity assertion in Theorem~\ref{thm:mod4} is also an immediate consequence of Andrews' formula for $v_d(q)$ in~\cite[Theorem~1, (1.4)]{AndrewsConcaveConvex}.
In our notation that formula reads
\[
S_1(q)=-\Theta(q)+2(-q;q)_\infty^2
\sum_{n\ge 0}q^{n(3n+1)/2}(1-q^{2n+1}),
\]
where
\[
\Theta(q)=\sum_{r\ge 0}(-1)^rq^{r(r+1)/2}.
\]
Reducing modulo $2$ gives $S_1(q)\equiv\Theta(q)\pmod{2}$, and hence $s_1(n)$ is odd exactly when $n$ is triangular.
Theorem~\ref{thm:mod4} refines this parity statement to a modulo $4$ evaluation.
\end{remark}

Since $An+B=\frac{r(r+1)}{2}$ if and only if $8An+8B+1=(2r+1)^2$, Theorem~\ref{thm:mod4} yields the next criterion.
\begin{corollary}\label{cor:progressions}
Let $A\ge 1$ and $0\le B<A$. Then
\[
s_1(An+B)\equiv 0\pmod{4}\qquad(n\ge 0)
\]
if and only if $8B+1$ is not a square modulo $8A$.
\end{corollary}

For odd primes $p$, the congruence $8B+1\equiv 1\pmod{8}$ shows that the square condition modulo $8p$ is equivalent to the Legendre-symbol condition modulo $p$.
\begin{corollary}\label{cor:prime-modulus}
Let $p$ be an odd prime. Then
\[
s_1(pn+B)\equiv 0\pmod{4}\qquad(n\ge 0)
\]
if and only if
\[
\left(\frac{8B+1}{p}\right)=-1.
\]
Hence for each odd prime $p$ there are exactly $(p-1)/2$ such progressions modulo $p$.
\end{corollary}

For example, Corollary~\ref{cor:prime-modulus} gives
\[
s_1(3n+2)\equiv 0\pmod{4},
\qquad
s_1(5n+2)\equiv s_1(5n+4)\equiv 0\pmod{4},
\]
and
\[
s_1(7n+2)\equiv s_1(7n+4)\equiv s_1(7n+5)\equiv 0\pmod{4}.
\]
The odd and even companion series are
\[
T_o(q):=\sum_{m\ge 0} q^{2m}\frac{\qpochn{-q}{q}{2m}}{\qpochn{q}{q^2}{m+1}},
\qquad
T_e(q):=\sum_{m\ge 0}\frac{\qpochn{-1}{q}{2m}q^{2m}}{\qpochn{q^2}{q^2}{m}}.
\]
They arise naturally from the odd-even decomposition underlying our formulas for $S_1(q)$; see Proposition~\ref{prop:companions} below.

An important series in this paper is the eta-normalized series $C(q)$ defined by
\begin{equation}\label{eta-normal}
C(q)=(q;q)_\infty T_o(q).
\end{equation}
Writing
\[
C(q)=\sum_{n\geq 0}c(n) q^n,
\]
the following main result shows that the sequence $\{c(n)\}$ has rich arithmetic properties.
\begin{theorem}\label{thm:C-recursions}
For all $n\ge 0$,
\[
c(25n+8)=c(25n+13)=c(25n+18)=c(25n+23)=0,
\]
and
\[
c(25n+28)=-c(n).
\]
In particular, $c(n)$ changes sign infinitely often and assumes both positive and negative values infinitely often.
\end{theorem}
Iterating the relation \(c(25n+28)=-c(n)\) therefore produces infinite towers of coefficients with alternating signs.

The next result shows that the support of $C(q)$ is constrained by the binary quadratic form $x^2+3y^2$,
and its corollary gives exact vanishing progressions modulo $p^2$ for every prime $p\equiv 2\pmod{3}$.

\begin{theorem}\label{thm:C-local}
If $c(n)\neq 0$, then
\[
24n+28=x^2+3y^2
\]
for some integers $x$ and $y$. Consequently, if $p\equiv 2\pmod{3}$ is an odd prime and $p$ divides $6n+7$ to odd order, then
\[
c(n)=0.
\]
\end{theorem}

\begin{corollary}\label{cor:C-local-progressions}
Let $p\equiv 2\pmod{3}$ be an odd prime. Let $r_p$ be the unique residue class modulo $p$ satisfying
\[
6r_p+7\equiv 0\pmod{p}.
\]
Among the $p$ lifts of $r_p$ modulo $p^2$, exactly one residue class satisfies
\[
6r+7\equiv 0\pmod{p^2}.
\]
For each of the remaining $p-1$ lifts $r$, we have
\[
c(p^2n+r)=0\qquad(n\ge 0).
\]
\end{corollary}

For $p=5$, Corollary~\ref{cor:C-local-progressions} yields the four vanishing progressions in Theorem~\ref{thm:C-recursions}. For $p=11$, it gives
\[
c(121n+r)=0\qquad(n\ge 0),
\]
for
\[
r\in\{8,30,41,52,63,74,85,96,107,118\}.
\]
The remainder of the paper will be organized as follows. The partition-theoretic results for the even companion and the comparison with OEIS~A344650 are collected in Section~\ref{sec:partition-identities}. The analytic formulas for the companion series are collected in Section~\ref{sec:analytic}. The proofs are given in Sections~\ref{sec:proof-thm1}--\ref{sec:proof-A344650}, and the paper closes in Section~\ref{sec:remarks} with final remarks, problems, and conjectures.
\section{$q$-Series Background}\label{sec:background}
In this section we collect several well-known facts from the theory of basic hypergeometric series that will be needed in our proofs.
We start with the $q$-binomial theorem~\cite[Eq. (2.2.1)]{Andrews}
\begin{equation}\label{q-binomial}
\sum_{n\geq 0}\fr{(a;q)_n z^n}{(q;q)_n} = \fr{(az;q)_\infty}{(z;q)_\infty},
\end{equation}
Euler's identity~\cite[Eq. (1.2.5)]{Andrews}
\begin{equation}\label{Euler0}
(-q;q)_\infty = \fr{1}{(q;q^2)_\infty},
\end{equation}
Euler's expansion~\cite[Eq. (2.2.6)]{Andrews}
\begin{equation}\label{Euler}
\qpoch{z}{q}=\sum_{n\ge 0}\frac{(-1)^n q^{n(n-1)/2}z^n}{\qpochn{q}{q}{n}},
\end{equation}
and Gauss' sum~\cite[Eq.~(2.2.13)]{Andrews}
\begin{equation}\label{Gauss-sum}
\fr{(q^2;q^2)_\infty}{(q;q^2)_\infty}=\sum_{n\ge 0}q^{n(n+1)/2}.
\end{equation}
We need the following false theta evaluation~\cite[Eq. (9.4.14)]{AndrewsBerndtI}
\begin{equation}\label{falsetheta}
\sum_{n\ge 0}\frac{(-q)^n}{\qpochn{-q^2}{q^2}{n}}
=\sum_{n\ge 0} q^{n(3n+1)/2}(1-q^{2n+1})
\end{equation}
and Euler's pentagonal theorem~\cite[Eq. (1.3.1)]{Andrews}
\begin{equation}\label{pentagonal}
(q;q)_\infty =\sum_{n\ge 0}(-1)^n q^{n(3n+1)/2}(1-q^{2n+1}).
\end{equation}
In addition, we will need Jacobi's triple product~\cite[Eq. (2.2.10)]{Andrews}
\begin{equation}\label{Jacobi-triple}
\sum_{n=-\infty}^{\infty} z^n q^{n^2}=(-zq,-q/z,q^2;q^2)_\infty
\end{equation}

Finally, recall the definition of the basic hypergeometric series~\cite[Eq.~(1.2.22)]{GasperRahman}:
\[
  \pFq{\,r+1\,}{\,r\,}
      {a_1, a_2, \dots, a_{r+1}}
      {b_1, b_2, \dots, b_r}
      {q; z}
  =
  \sum_{n=0}^{\infty}
    \frac{(a_1;q)_n (a_2;q)_n \cdots (a_{r+1};q)_n}
         {(q;q)_n (b_1;q)_n \cdots (b_r;q)_n}
    \, z^n.
\]
Heine's first transformation~\cite[(III.1)]{GasperRahman} states that
\begin{equation}\label{Heine-1}
{}_2\phi_1\!\left(\begin{matrix}a,b\\ c\end{matrix};q,z\right)
=\frac{\qpoch{b,az}{q}}{\qpoch{c,z}{q}}
{}_2\phi_1\!\left(\begin{matrix}c/b,z\\ az\end{matrix};q,b\right).
\end{equation}
\section{Partition identities and strict odd-length partitions}\label{sec:partition-identities}
The coefficients of the even companion admit a direct overpartition interpretation. Recall that an overpartition is a partition in which the first occurrence of a part may be overlined; see Corteel and Lovejoy~\cite{CorteelLovejoy}. Write
\[
T_e(q)=1+\sum_{N\ge 1} t_e(N)q^N.
\]
For $m\ge 1$, the $m$th summand in the defining series for $T_e(q)$ can be written as
\[
2\frac{q^{2m}}{1-q^{2m}}
\prod_{j=1}^{m-1}\frac{1+q^{2j}}{1-q^{2j}}
\prod_{j=1}^{m}(1+q^{2j-1}).
\]
Thus, for $N\ge 1$, the coefficient $t_e(N)$ counts overpartitions of $N$ such that the largest part is even, only even parts may be overlined, and the odd parts are distinct; the factor $2q^{2m}/(1-q^{2m})$ records the two choices for the overline status of the largest part $2m$. In particular, $t_e(N)$ is even for every $N\ge 1$.

\begin{definition}\label{def:A-B}
For $N\ge 1$, let $a_e(N)$ denote the number of ordinary partitions of $N$ whose first missing part among the positive integers not congruent to $1\pmod{4}$ is congruent to $3\pmod{4}$. Equivalently, among the candidate parts
\[
2,3,4,6,7,8,10,11,12,\dots,
\]
the smallest candidate that does not occur is of the form $4m+3$.

For $N\ge 1$, let $b_e(N)$ denote the number of overpartitions of $N$ such that the largest part is even and not overlined, only even parts may be overlined, and the odd parts are distinct.
\end{definition}

The direct interpretation of \(t_e(N)\) above shows immediately that
\begin{equation}\label{te-be}
t_e(N)=2b_e(N)\qquad(N\ge 1).
\end{equation}

\begin{remark}\label{rem:Te-example}
For $N=6$, the eight overpartitions counted by $t_e(6)$ are
\[
(6),\qquad (\overline{6}),\qquad (4,2),\qquad (\overline{4},2),\qquad (4,\overline{2}),\qquad (\overline{4},\overline{2}),\qquad (2,2,2),\qquad (\overline{2},2,2).
\]
Thus $t_e(6)=8$.

Restricting to the overpartitions with nonoverlined largest part leaves the four objects counted by $b_e(6)$:
\[
(6),\qquad (4,2),\qquad (4,\overline{2}),\qquad (2,2,2).
\]
Hence $b_e(6)=4$.

The ordinary partitions counted by $a_e(6)$ are
\[
(4,2),\qquad (2,2,2),\qquad (2,2,1,1),\qquad (2,1,1,1,1).
\]
In each case the relevant candidate parts begin
\[
2,3,4,6,7,8,\dots,
\]
and the first missing candidate is $3$. Thus $a_e(6)=b_e(6)=t_e(6)/2=4$.
\end{remark}

\begin{theorem}\label{thm:Te-partitions}
For every $N\ge 1$,
\[
a_e(N)=b_e(N)=\frac{t_e(N)}{2}.
\]
\end{theorem}

We now turn to strict odd-length partitions and OEIS A344650.

OEIS A344650 is the sequence whose $n$-th term counts the number of strict partitions of $2n$ into an odd number of parts; equivalently, it counts partitions of $2n$ into an odd number of distinct parts~\cite{OEISA344650}. We denote this sequence by $a(n)$ and write
\[
A(q):=\sum_{n\ge 0} a(n)q^n.
\]
The comparison with $T_o(q)$ is given later in Remark~\ref{rem:A344650-comparison}.

For partitions into distinct parts, let $p_d(k,n)$ denote the number of partitions of $n$ into exactly $k$ distinct positive parts.

\begin{definition}\label{def:Podd}
Define
\[
p_{\mathrm{odd}}(n):=\sum_{r\ge 0} p_d(2r+1,n),
\qquad
P_{\mathrm{odd}}(q):=\sum_{n\ge 0} p_{\mathrm{odd}}(n)q^n.
\]
Thus $p_{\mathrm{odd}}(n)$ counts partitions of $n$ into an odd number of distinct parts, and we clearly have
\begin{equation}\label{a-podd}
a(n)=p_{\mathrm{odd}}(2n)\qquad(n\ge 0).
\end{equation}
\end{definition}

For example, when $n=8$, the only admissible odd lengths are $1$ and $3$, since the smallest strict partition of odd length $5$ is $1+2+3+4+5>8$. The strict partitions of $8$ into an odd number of parts are
\[
(8),\qquad (5,2,1),\qquad (4,3,1).
\]
Therefore
\begin{align*}
p_{\mathrm{odd}}(8)&=3,\\
A(q)&=q+q^2+2q^3+3q^4+\cdots,\\
a(4)&=p_{\mathrm{odd}}(8)=3.
\end{align*}

\begin{theorem}\label{thm:A344650}
We have
\begin{align}
P_{\mathrm{odd}}(q)
&=\sum_{r\ge 0}\frac{q^{(2r+1)(2r+2)/2}}{\qpochn{q}{q}{2r+1}} \label{Podd-1} \\
&=\frac12\bigl(\qpoch{-q}{q}-\qpoch{q}{q}\bigr) \label{Podd-2} \\
A(q)
&=\frac{1}{\qpoch{q}{q}}
\left(\sum_{m\in\mathbb Z} q^{4m^2-m}\right)
\left(\sum_{m\ge 1}(-1)^{m-1}q^{m^2}\right). \label{A344650-A}
\end{align}
\end{theorem}

\begin{remark}\label{rem:A344650-comparison}
We now make explicit the comparison between $T_o(q)$ and OEIS A344650. By Theorem~\ref{thm:A344650},
\[
\frac{(q;q)_\infty}{q}A(q)=1+q^5-q^6+q^{11}+q^{13}+q^{14}-q^{15}+O(q^{16}).
\]
A direct expansion gives
\[
C(q)=1+q^5-q^6+q^{11}+q^{14}-q^{15}+O(q^{16}).
\]
Hence $C(q)$ and $\frac{(q;q)_\infty}{q}A(q)$ agree through $q^{12}$ and first differ at $q^{13}$. Because $(q;q)_\infty^{-1}=1+O(q)$, the same is true of $T_o(q)$ and $q^{-1}A(q)$. In particular,
\[
[q^{13}]T_o(q)=110,
\qquad
 a(14)=111.
\]
\end{remark}
\section{Analytic formulas for the companion series}\label{sec:analytic}
In this section we collect the analytic identities for the companion series and for the eta-normalized series
\[
C(q):=(q;q)_\infty T_o(q).
\]

For convenience, write
\begin{align}
\Theta(q)&:=\sum_{r\ge 0}(-1)^r q^{r(r+1)/2},\label{Theta-def} \\
\Psi(q)&:=\frac{\qpoch{q^2}{q^2}}{\qpoch{q}{q^2}}.\label{Psi-def}
\end{align}
Note that by Gauss' sum~\eqref{Gauss-sum} we have
\begin{equation}
\Psi(q)=\sum_{n\ge 0}q^{n(n+1)/2}.
\end{equation}

\begin{proposition}\label{prop:companions}
We have
\begin{align}
S_1(q) &=\Psi(q)T_e(q)+2qT_o(q), \label{S1-id1}\\
&=\Theta(q)+4qT_o(q), \label{S1-id2} \\
\Psi(q)T_e(q)&=\Theta(q)+2qT_o(q). \label{cor:companions}
\end{align}
\end{proposition}

Define
\[
P(q):=\sum_{n\ge 0} q^{n(3n+1)/2}(1-q^{2n+1}).
\]
The next proposition will be used in the proof of Theorem~\ref{thm:Te-partitions}.

\begin{proposition}\label{prop:Te-closed}
We have
\begin{equation}\label{Te-closed}
T_e(q)=\frac{P(q)}{\qpoch{q}{q}}.
\end{equation}
\end{proposition}

Combining Proposition~\ref{prop:companions} with Proposition~\ref{prop:Te-closed}, we obtain
\begin{equation}\label{C-basic}
2qC(q)=\Psi(q)P(q)-(q;q)_\infty\Theta(q).
\end{equation}

\begin{remark}\label{rem:andrews-vd-proposition}
The first identity \eqref{S1-id1} is the direct even--odd decomposition of $S_1(q)$.
In conjunction with this decomposition, the remaining identities \eqref{S1-id2} and \eqref{cor:companions} can also be recovered from Andrews' identity for $v_d(q)$.
Since $S_1(q)=v_d(q)$, \cite[Theorem~1, (1.4)]{AndrewsConcaveConvex} gives
\begin{equation}\label{andrews-vd-formula}
S_1(q)=-\Theta(q)+2(-q;q)_\infty^2P(q).
\end{equation}
On the other hand, Proposition~\ref{prop:Te-closed}, Euler's identity \eqref{Euler0}, and
\[
(q;q)_\infty=(q;q^2)_\infty(q^2;q^2)_\infty
\]
give
\[
(-q;q)_\infty^2P(q)
=\frac{P(q)}{(q;q^2)_\infty^2}
=\frac{(q^2;q^2)_\infty}{(q;q^2)_\infty}\frac{P(q)}{(q;q)_\infty}
=\Psi(q)T_e(q).
\]
Thus \eqref{andrews-vd-formula} is equivalent to
\[
S_1(q)=2\Psi(q)T_e(q)-\Theta(q).
\]
Comparing this with the direct even--odd decomposition \eqref{S1-id1} gives
\[
\Psi(q)T_e(q)=\Theta(q)+2qT_o(q),
\]
which is \eqref{cor:companions}; substituting this back into \eqref{S1-id1} gives
\[
S_1(q)=\Theta(q)+4qT_o(q),
\]
which is \eqref{S1-id2}.
This provides an alternative derivation of these two companion identities from Andrews' formula, while the proof below remains self-contained.
\end{remark}

The next proposition gives a parity-split representation of $C(q)$.
\begin{proposition}\label{prop:C-structure}
The eta-normalized series $C(q)$ has the representation
\begin{align}
C(q)
&=
\left(\sum_{n\ge 0} q^{6n^2+n}(1-q^{4n+1})\right)
\left(\sum_{n\ge 0} q^{2n^2+3n}\right) \notag\\
&\quad+
\left(\sum_{n\ge 0} q^{6n^2+7n+1}(1-q^{4n+3})\right)
\left(\sum_{n\ge 0} q^{2n^2+n}\right). \label{C-parity-2}
\end{align}
\end{proposition}

\begin{remark}\label{rem:x2-plus-3y2}
The quadratic form $x^2+3y^2$ that appears in Theorem~\ref{thm:C-local} is naturally attached to the theta series
\[
\sum_{x,y\in\mathbb Z} q^{x^2+3y^2},
\]
a classical modular form of weight $1$. Thus Theorem~\ref{thm:C-local} places the support of $C(q)$ in a classical modular-form setting without asserting modularity for $C(q)$ itself.
\end{remark}
\section{Proof of Proposition~\ref{prop:companions} and Theorem~\ref{thm:mod4}}\label{sec:proof-thm1}
We first need a lemma.
\begin{lemma}\label{lem:H-theta}
Let
\[
H(q):=\sum_{n\ge 0} q^n\,\qpoch{q^{n+1}}{q}^{2}.
\]
Then we have
\[
H(q)=\Theta(q).
\]
\end{lemma}

\begin{proof}
Consider
\[
G(z):=\sum_{n\ge 0} z^n\,\qpoch{q^{n+1}}{q}^{2}.
\]
Then
\begin{equation}\label{H-help1}
G(q) = H(q).
\end{equation}
By the $q$-binomial theorem~\eqref{q-binomial}
applied with $a=0$ and Euler's expansion~\eqref{Euler} with $z=q^{n+1}$, we obtain
\begin{align*}
G(z)
&=\qpoch{q}{q}\sum_{n\ge 0}\frac{z^n}{\qpochn{q}{q}{n}}\,\qpoch{q^{n+1}}{q} \\
&=\qpoch{q}{q}\sum_{r\ge 0}\frac{(-1)^r q^{r(r+1)/2}}{\qpochn{q}{q}{r}}
   \sum_{n\ge 0}\frac{(zq^r)^n}{\qpochn{q}{q}{n}} \\
&=\frac{\qpoch{q}{q}}{\qpoch{z}{q}}
  \sum_{r\ge 0}\frac{(-1)^r q^{r(r+1)/2}\qpochn{z}{q}{r}}{\qpochn{q}{q}{r}}.
\end{align*}
Now letting $z=q$ in the previous identity and using~\eqref{H-help1}, we obtain
\[
H(q)=\sum_{r\ge 0}(-1)^r q^{r(r+1)/2}=\Theta(q).
\]
This completes the proof.
\end{proof}

\begin{proof}[Proof of Proposition~\ref{prop:companions}]
We first prove~\eqref{S1-id1}. From
\[
\qpoch{-q^{n+1}}{q}^{2}
=\qpoch{q^{2n+2}}{q^2}
\frac{\qpoch{-q^{n+1}}{q}}{\qpoch{q^{n+1}}{q}},
\]
and~\eqref{q-binomial} with $z=q^{n+1}$, we get
\begin{align*}
S_1(q) &=\sum_{n\ge 0}q^n\qpoch{q^{2n+2}}{q^2}
\frac{\qpoch{-q^{n+1}}{q}}{\qpoch{q^{n+1}}{q}} \\
&=\qpoch{q^2}{q^2}\sum_{n\ge 0}\frac{q^n}{\qpochn{q^2}{q^2}{n}}
   \sum_{k\ge 0}\frac{\qpochn{-1}{q}{k}q^{(n+1)k}}{\qpochn{q}{q}{k}} \\
&=\qpoch{q^2}{q^2}\sum_{k\ge 0}\frac{\qpochn{-1}{q}{k}q^k}{\qpochn{q}{q}{k}}
   \sum_{n\ge 0}\frac{q^{n(k+1)}}{\qpochn{q^2}{q^2}{n}} \\
&=\qpoch{q^2}{q^2}\sum_{k\ge 0}
\frac{\qpochn{-1}{q}{k}q^k}{\qpochn{q}{q}{k}\qpoch{q^{k+1}}{q^2}},
\end{align*}
where in the last step we used~\eqref{q-binomial} with $q\to q^2$ and $z=q^{k+1}$.

Separate the sum into even and odd values of $k$.
For $k=2m$, we have
\[
\qpochn{q}{q}{2m}=\qpochn{q}{q^2}{m}\qpochn{q^2}{q^2}{m}
\]
and
\[
\frac{\qpoch{q^2}{q^2}}{\qpoch{q^{2m+1}}{q^2}}
=\frac{\qpoch{q^2}{q^2}}{\qpoch{q}{q^2}}\qpochn{q}{q^2}{m}
=\Psi(q)\qpochn{q}{q^2}{m}.
\]
Hence the even part is
\[
\Psi(q)\sum_{m\ge 0}
\frac{\qpochn{-1}{q}{2m}q^{2m}}{\qpochn{q^2}{q^2}{m}}
=\Psi(q)T_e(q).
\]
For $k=2m+1$, we use
\[
\qpochn{-1}{q}{2m+1}=2\qpochn{-q}{q}{2m},
\qquad
\qpochn{q}{q}{2m+1}=\qpochn{q}{q^2}{m+1}\qpochn{q^2}{q^2}{m},
\]
and
\[
\frac{\qpoch{q^2}{q^2}}{\qpoch{q^{2m+2}}{q^2}}=\qpochn{q^2}{q^2}{m}.
\]
Therefore the odd part is
\[
2q\sum_{m\ge 0}q^{2m}
\frac{\qpochn{-q}{q}{2m}}{\qpochn{q}{q^2}{m+1}}
=2qT_o(q).
\]
Combining the even and odd parts gives
\[
S_1(q)=\Psi(q)T_e(q)+2qT_o(q),
\]
which is~\eqref{S1-id1}.

To prove~\eqref{S1-id2}, apply the same calculation to
\[
H(q):=\sum_{n\ge 0}q^n\qpoch{q^{n+1}}{q}^{2}.
\]
Since
\[
\qpoch{q^{n+1}}{q}^{2}
=\qpoch{q^{2n+2}}{q^2}
\frac{\qpoch{q^{n+1}}{q}}{\qpoch{-q^{n+1}}{q}},
\]
and
\[
\frac{\qpoch{z}{q}}{\qpoch{-z}{q}}
=\sum_{k\ge 0}\frac{\qpochn{-1}{q}{k}}{\qpochn{q}{q}{k}}(-z)^k
=\sum_{k\ge 0}\frac{(-1)^k\qpochn{-1}{q}{k}}{\qpochn{q}{q}{k}}z^k,
\]
we obtain
\[
H(q)=\qpoch{q^2}{q^2}
\sum_{k\ge 0}
\frac{(-1)^k\qpochn{-1}{q}{k}q^k}{\qpochn{q}{q}{k}\qpoch{q^{k+1}}{q^2}}.
\]
Subtract this identity from the corresponding formula for $S_1(q)$. The terms with even $k$ cancel, and only the odd terms $k=2m+1$ remain:
\begin{align*}
S_1(q)-H(q)
&=2\qpoch{q^2}{q^2}
  \sum_{m\ge 0}
\frac{\qpochn{-1}{q}{2m+1}q^{2m+1}}{\qpochn{q}{q}{2m+1}\qpoch{q^{2m+2}}{q^2}} \\
&=4q\sum_{m\ge 0}q^{2m}
\frac{\qpochn{-q}{q}{2m}}{\qpochn{q}{q^2}{m+1}} \\
&=4qT_o(q).
\end{align*}
By Lemma~\ref{lem:H-theta}, $H(q)=\Theta(q)$, and hence
\[
S_1(q)=\Theta(q)+4qT_o(q).
\]
This proves~\eqref{S1-id2}.

Finally, subtracting the two identities~\eqref{S1-id1} and~\eqref{S1-id2} yields~\eqref{cor:companions}.
\end{proof}

\begin{proof}[Proof of Theorem~\ref{thm:mod4}]
Reducing~\eqref{S1-id2} modulo $4$ gives
\[
S_1(q)\equiv \Theta(q)=\sum_{r\ge 0}(-1)^r q^{r(r+1)/2}\pmod{4}.
\]
Comparing coefficients yields the stated formula for $s_1(n)$.
\end{proof}
\section{Proof of Theorem~\ref{thm:C-recursions}}\label{sec:proof-C-recursions}
We need two lemmas.

Throughout, if $F(q)=\sum_{n\ge 0} f(n)q^n$ and $r\in\{0,1,2,3,4\}$, write
\[
\bigl(F(q)\bigr)^{(r)}:=\sum_{\substack{n\ge 0\\ n\equiv r\pmod{5}}} f(n)q^n.
\]
for the subseries of $F(q)$ supported on exponents congruent to $r\pmod{5}$.

\begin{lemma}\label{lem:five-parts}
The following residue extractions hold:
\[
\bigl( (q;q)_\infty \bigr)^{(1)}=-q(q^{25};q^{25})_\infty,
\]
\[
\bigl(\Psi(q)\bigr)^{(3)}=q^3\Psi(q^{25}),
\qquad
\bigl(\Theta(q)\bigr)^{(3)}=q^3\Theta(q^{25}),
\]
and
\[
\bigl(P(q) \bigr)^{(1)}=-qP(q^{25}).
\]
\end{lemma}

\begin{proof}
By Euler's pentagonal theorem in bilateral form,
\[
(q;q)_\infty=\sum_{m\in\mathbb Z}(-1)^m q^{m(3m-1)/2}.
\]
The exponents congruent to $1\pmod{5}$ occur exactly when $m\equiv 1\pmod{5}$. Writing $m=5k+1$ gives
\[
\bigl((q;q)_\infty\bigr)^{(1)}=-q\sum_{k\in\mathbb Z}(-1)^k q^{25k(3k+1)/2}.
\]
Replacing $k$ by $-k$ and applying Euler's pentagonal theorem again yields the first identity.

Since $n(n+1)/2\equiv 0,1,3\pmod{5}$, and residue $3$ occurs only when $n\equiv 2\pmod{5}$, we obtain with the help of~\eqref{Gauss-sum}
\[
\bigl(\Psi(q)\bigr)^{(3)}=\sum_{k\ge 0} q^{(5k+2)(5k+3)/2}=q^3\Psi(q^{25}).
\]
The same residue condition together with the factor $(-1)^n$ gives
\[
\bigl(\Theta(q)\bigr)^{(3)}=\sum_{k\ge 0} (-1)^{5k+2}q^{(5k+2)(5k+3)/2}=q^3\Theta(q^{25}).
\]

Finally,
\[
P(q)=\sum_{n\ge 0}q^{n(3n+1)/2}-\sum_{n\ge 0}q^{n(3n+1)/2+2n+1}.
\]
The residue class $1\pmod{5}$ occurs in the first sum only when $n\equiv 4\pmod{5}$, and in the second sum only when $n\equiv 0\pmod{5}$. Writing $n=5k+4$ and $n=5k$ gives $\bigl(P\bigr)^{(1)}(q)=-qP(q^{25})$.
\end{proof}
The following lemma states that the coefficients of $C(q)$ satisfy the exact quintic self-similarity.
\begin{lemma}\label{lem:C-quintic}
We have
\begin{equation}\label{C-quintic}
\sum_{n\ge 0} c(5n+3)q^n=-q^5C(q^5).
\end{equation}
\end{lemma}
\begin{proof}
Take the residue-$4$ part of~\eqref{C-basic}. Since $\Psi(q)$ only has residues $0,1,3\pmod{5}$ and $P(q)$ only has residues $0,1,2\pmod{5}$, the only contribution to residue $4$ in $\Psi(q)P(q)$ comes from the pair $(3,1)$. Likewise, the only contribution to residue $4$ in $(q;q)_\infty\Theta(q)$ comes from the pair $(1,3)$. Lemma~\ref{lem:five-parts} therefore gives
\begin{align*}
\bigl(2qC(q)\bigr)^{(4)}
&=\bigl(\Psi(q)\bigr)^{(3)} \bigl(P(q)\bigr)^{(1)}-\bigl((q;q)_\infty\bigr)^{(1)} \bigl(\Theta(q)\bigr)^{(3)} \\
&=q^3\Psi(q^{25})\cdot\bigl(-qP(q^{25})\bigr)
   -\bigl(-q(q^{25};q^{25})_\infty\bigr)\cdot q^3\Theta(q^{25}) \\
&=q^4\Bigl((q^{25};q^{25})_\infty\Theta(q^{25})-\Psi(q^{25})P(q^{25})\Bigr).
\end{align*}
Now apply~\eqref{C-basic} with $q$ replaced by $q^{25}$:
\[
2q^{25}C(q^{25})=\Psi(q^{25})P(q^{25})-(q^{25};q^{25})_\infty\Theta(q^{25}).
\]
Therefore
\[
\bigl(2qC(q) \bigr)^{(4)}=-2q^{29}C(q^{25}).
\]
On the other hand,
\[
2qC(q)=2\sum_{m\ge 0}c(m)q^{m+1},
\]
so its $4$-part is
\[
\bigl(2qC(q) \bigr)^{(4)}=2\sum_{n\ge 0}c(5n+3)q^{5n+4}.
\]
Comparing the last two identities, dividing by $2q^4$, and replacing $q^5$ by $q$ gives the claim.
\end{proof}
We are now in a position to establish Theorem~\ref{thm:C-recursions}.

\begin{proof}[Proof of Theorem~\ref{thm:C-recursions}]
By~\eqref{C-quintic}, the right-hand side is supported on positive exponents divisible by $5$. Hence the coefficient of $q^n$ on the left vanishes whenever $n\not\equiv 0\pmod{5}$, which is equivalent to
\[
c(25n+8)=c(25n+13)=c(25n+18)=c(25n+23)=0
\qquad(n\ge 0).
\]
If $n\ge 1$, then comparing coefficients of $q^{5n}$ in~\eqref{C-quintic} yields
\[
c(25n+3)=-c(n-1).
\]
Replacing $n$ by $n+1$ gives the equivalent form $c(25n+28)=-c(n)$ for $n\ge 0$.

Finally, Proposition~\ref{prop:C-structure} shows that $c(0)=1$ and $c(28)=-1$. Iterating the relation $c(25n+28)=-c(n)$ therefore produces infinitely many positive values and infinitely many negative values, so the signs change infinitely often.
\end{proof}
\section{Proof of Proposition~\ref{prop:Te-closed} and Proposition~\ref{prop:C-structure}}\label{sec:proof-Te-C}
\begin{proof}[Proof of Proposition~\ref{prop:Te-closed}]
Write
\begin{align*}
T_e(q) &= \sum_{n\geq 0}\fr{(-1;q)_{2n}}{(q^2;q^2)_n} q^{2n} \\
&=\sum_{n\geq 0}\fr{(-1;q^2)_n (-q;q^2)_n}{(q^2;q^2)_n} q^{2n} \\
&= {}_2\phi_1\!\left(\begin{matrix}-1,-q\\ 0\end{matrix};q^2,q^2\right).
\end{align*}
Then~\eqref{Heine-1}, with $q\to q^2$ and $(a,b,c,z)=(-1,-q,0,q^2)$, gives
\[
T_e(q)=\frac{\qpoch{-q,-q^2}{q^2}}{\qpoch{q^2}{q^2}}
{}_2\phi_1\!\left(\begin{matrix}0,q^2\\ -q^2\end{matrix};q^2,-q\right)
=\frac{1}{\qpoch{q}{q}}\sum_{n\ge 0}\frac{(-q)^n}{\qpochn{-q^2}{q^2}{n}}.
\]
Now combine the previous identity with~\eqref{falsetheta} to deduce that
\begin{align*}
T_e(q)&=\frac{1}{\qpoch{q}{q}} \sum_{n\ge 0} q^{n(3n+1)/2}(1-q^{2n+1}) \\
&=\frac{P(q)}{\qpoch{q}{q}},
\end{align*}
which proves Proposition~\ref{prop:Te-closed}.
\end{proof}

\begin{proof}[Proof of Proposition~\ref{prop:C-structure}]
Set
\[
A_n:=q^{n(3n+1)/2}(1-q^{2n+1}),
\qquad
B_n:=q^{n(n+1)/2}.
\]
Then
\[
(q;q)_\infty=\sum_{n\ge 0}(-1)^nA_n,
\qquad
P(q)=\sum_{n\ge 0}A_n,
\]
\[
\Theta(q)=\sum_{n\ge 0}(-1)^nB_n,
\qquad
\Psi(q)=\sum_{n\ge 0}B_n.
\]
Therefore \eqref{C-basic} can be rewritten as
\[
2qC(q)
= -\left(\sum_{n\ge 0}(-1)^nA_n\right)\left(\sum_{n\ge 0}(-1)^nB_n\right)
  +\left(\sum_{n\ge 0}A_n\right)\left(\sum_{n\ge 0}B_n\right).
\]
Now write
\[
A_{\mathrm e}:=\sum_{n\ge 0}A_{2n},\quad A_{\mathrm o}:=\sum_{n\ge 0}A_{2n+1},
\qquad
B_{\mathrm e}:=\sum_{n\ge 0}B_{2n},\quad B_{\mathrm o}:=\sum_{n\ge 0}B_{2n+1}.
\]
Then
\[
\begin{aligned}
2qC(q)
&=-(A_{\mathrm e}-A_{\mathrm o})(B_{\mathrm e}-B_{\mathrm o})
  +(A_{\mathrm e}+A_{\mathrm o})(B_{\mathrm e}+B_{\mathrm o})\\
&=2A_{\mathrm e}B_{\mathrm o}+2A_{\mathrm o}B_{\mathrm e}.
\end{aligned}
\]
Finally,
\[
A_{2n}=q^{n(6n+1)}(1-q^{4n+1}),
\qquad
A_{2n+1}=q^{(2n+1)(3n+2)}(1-q^{4n+3}),
\]
and
\[
B_{2n}=q^{n(2n+1)},
\qquad
B_{2n+1}=q^{(n+1)(2n+1)}.
\]
Substituting these formulas gives
\begin{align}
qC(q)
&=
\left(\sum_{n\ge 0} q^{n(6n+1)}(1-q^{4n+1})\right)
\left(\sum_{n\ge 0} q^{(n+1)(2n+1)}\right) \notag\\
&\quad+
\left(\sum_{n\ge 0} q^{(2n+1)(3n+2)}(1-q^{4n+3})\right)
\left(\sum_{n\ge 0} q^{n(2n+1)}\right). \label{C-parity-1}
\end{align}

Hence dividing~\eqref{C-parity-1} by $q$ gives~\eqref{C-parity-2}.
\end{proof}
\section{Proof of Theorem~\ref{thm:C-local} and Corollary~\ref{cor:C-local-progressions}}\label{sec:proof-C-local}

\begin{proof}[Proof of Theorem~\ref{thm:C-local}]
By Proposition~\ref{prop:C-structure}, every exponent $n$ with $c(n)\ne 0$ arises from one of the four shapes
\begin{align*}
n&=6u^2+u+2v^2+3v, \\
n&=6u^2+5u+1+2v^2+3v, \\
n&=6u^2+7u+1+2v^2+v, \\
n&=6u^2+11u+4+2v^2+v,
\end{align*}
for some $u,v\ge 0$. Equivalently,
\begin{align*}
24n+28&=(12u+1)^2+3(4v+3)^2, \\
24n+28&=(12u+5)^2+3(4v+3)^2, \\
24n+28&=(12u+7)^2+3(4v+1)^2, \\
24n+28&=(12u+11)^2+3(4v+1)^2.
\end{align*}
Hence $24n+28$ is represented by the binary quadratic form $x^2+3y^2$.

Now let $p\equiv 2\pmod{3}$ be an odd prime and suppose that
\[
24n+28=x^2+3y^2.
\]
If $p\mid x^2+3y^2$, then modulo $p$ we have $x^2\equiv -3y^2$. Since
\[
\left(\frac{-3}{p}\right)=-1
\]
for every odd prime $p\equiv 2\pmod{3}$, this forces $p\mid x$ and $p\mid y$, and hence $p^2\mid x^2+3y^2$.
Therefore every odd prime $p\equiv 2\pmod{3}$ occurs to even exponent in $24n+28=4(6n+7)$. It follows that if such a prime $p$ divides $6n+7$ to odd order, then $c(n)=0$.
\end{proof}

\begin{proof}[Proof of Corollary~\ref{cor:C-local-progressions}]
Since $p$ is an odd prime different from $3$, the congruence
\[
6r+7\equiv 0\pmod{p}
\]
has a unique solution modulo $p$, namely $r\equiv r_p\pmod{p}$. Among its $p$ lifts modulo $p^2$, exactly one also satisfies
\[
6r+7\equiv 0\pmod{p^2}.
\]
For every other lift $r$, we have $p\mid 6r+7$ but $p^2\nmid 6r+7$, so $p$ divides
\[
6(p^2n+r)+7
\]
to odd order for every $n\ge 0$. Theorem~\ref{thm:C-local} therefore gives
\[
c(p^2n+r)=0\qquad(n\ge 0),
\]
as claimed.
\end{proof}
\section{Proof of Theorem~\ref{thm:Te-partitions}}\label{sec:proof-Te-part}
\begin{proof}[Proof of Theorem~\ref{thm:Te-partitions}]
Since
\[
P(q)=\sum_{n\ge 0} q^{n(3n+1)/2}(1-q^{2n+1})
\]
and, by~\eqref{pentagonal},
\[
(q;q)_\infty =\sum_{n\ge 0}(-1)^n q^{n(3n+1)/2}(1-q^{2n+1}),
\]
we derive
\[
P(q)-\qpoch{q}{q}
=2\sum_{n\ge 0} q^{(2n+1)(3n+2)}(1-q^{4n+3}).
\]
Combining this identity with Proposition~\ref{prop:Te-closed}, we obtain
\begin{equation}\label{Te-minus-one}
\frac{T_e(q)-1}{2}
=\frac{P(q)-\qpoch{q}{q}}{2\qpoch{q}{q}}
=\frac{1}{\qpoch{q}{q}}
\sum_{n\ge 0} q^{(2n+1)(3n+2)}(1-q^{4n+3}).
\end{equation}
For the ordinary partition interpretation, the summand indexed by $n$ in~\eqref{Te-minus-one} can be written as
\[
\frac{q^{2+3+4+6+7+8+\cdots +(4n)+(4n+2)}}{\prod_{\substack{m\ge 1\\ m\ne 4n+3}}(1-q^m)}.
\]
It therefore forces the parts
\[
2,3,4,6,7,8,\dots,4n,4n+2
\]
to appear, forbids the part $4n+3$, and leaves all other parts unrestricted. Hence it counts partitions whose first missing part among the positive integers not congruent to $1\pmod{4}$ is exactly $4n+3$. Summing over $n\ge 0$ gives
\[
\sum_{N\ge 1} a_e(N)q^N=\frac{T_e(q)-1}{2}.
\]
Since
\[
T_e(q)=1+\sum_{N\ge 1} t_e(N)q^N,
\]
comparison of coefficients yields
\[
a_e(N)=\frac{t_e(N)}{2}\qquad(N\ge 1).
\]

For the overpartition interpretation, the discussion preceding Definition~\ref{def:A-B} shows that $t_e(N)$ counts overpartitions of $N$ whose largest part is even, only even parts may be overlined, and the odd parts are distinct. Toggling the overline on the first occurrence of the largest part is a free involution on this set, so exactly half of these overpartitions have nonoverlined largest part. Therefore
\[
b_e(N)=\frac{t_e(N)}{2}\qquad(N\ge 1).
\]
Combining the two identities, we obtain
\[
a_e(N)=b_e(N)=\frac{t_e(N)}{2}\qquad(N\ge 1),
\]
which proves Theorem~\ref{thm:Te-partitions}.
\end{proof}
\section{Proof of Theorem~\ref{thm:A344650}}\label{sec:proof-A344650}

We begin with the standard staircase lemma.

\begin{lemma}\label{lem:dk}
For each fixed $k\ge 1$,
\[
\sum_{n\ge 0} p_d(k,n)q^n=\frac{q^{k(k+1)/2}}{\qpochn{q}{q}{k}}.
\]
\end{lemma}

\begin{proof}
Write a partition into exactly $k$ distinct positive parts as
\[
\lambda_1>\lambda_2>\cdots>\lambda_k\ge 1.
\]
Remove the staircase $(k,k-1,\dots,1)$ by setting
\[
\mu_i=\lambda_i-(k-i+1)\qquad(1\le i\le k).
\]
Then
\[
\mu_1\ge \mu_2\ge \cdots\ge \mu_k\ge 0,
\]
so $\mu$ is an ordinary partition with at most $k$ parts. Conversely, adding the staircase back to any partition with at most $k$ parts gives a partition into exactly $k$ distinct parts. The staircase contributes $k(k+1)/2$, and the generating function for partitions with at most $k$ parts is $1/(q;q)_k$; see, for example, Andrews~\cite[Chapter~1]{Andrews}.
\end{proof}

\begin{proof}[Proof of Theorem~\ref{thm:A344650}]
By Definition~\ref{def:Podd} and Lemma~\ref{lem:dk},
\[
P_{\mathrm{odd}}(q)
=\sum_{n\ge 0} p_{\mathrm{odd}}(n)q^n
=\sum_{r\ge 0}\sum_{n\ge 0} p_d(2r+1,n)q^n
=\sum_{r\ge 0}\frac{q^{(2r+1)(2r+2)/2}}{\qpochn{q}{q}{2r+1}},
\]
which confirms~\eqref{Podd-1}.

Next, Euler's expansion~\eqref{Euler} gives
\begin{align*}
\qpoch{zq}{q}
&=\sum_{n\ge 0}\frac{(-1)^n q^{n(n+1)/2}z^n}{\qpochn{q}{q}{n}} \\
\qpoch{-zq}{q}
&=\sum_{n\ge 0}\frac{q^{n(n+1)/2}z^n}{\qpochn{q}{q}{n}}.
\end{align*}
Taking the odd part in $z$ gives
\[
\frac{\qpoch{-zq}{q}-\qpoch{zq}{q}}{2}
=\sum_{r\ge 0}\frac{q^{(2r+1)(2r+2)/2}z^{2r+1}}{\qpochn{q}{q}{2r+1}}
\]
and setting $z=1$, we obtain
\[
P_{\mathrm{odd}}(q)=\frac12\bigl(\qpoch{-q}{q}-\qpoch{q}{q}\bigr),
\]
which proves~\eqref{Podd-2}.

We now derive the formula for $A(q)$. By~\eqref{a-podd} and~\eqref{Podd-2}, we get
\begin{align}
A(q^2)
&=\frac{P_{\mathrm{odd}}(q)+P_{\mathrm{odd}}(-q)}{2} \nonumber\\
&=\frac14\Bigl(\qpoch{-q}{q}-\qpoch{q}{q}+\qpoch{q}{-q}-\qpoch{-q}{-q}\Bigr) \nonumber\\
&=\frac14\Bigl(\qpoch{-q}{q^2}\qpoch{-q^2}{q^2}
-\qpoch{q}{q^2}\qpoch{q^2}{q^2} \nonumber \\
&\qquad\quad
+\qpoch{q}{q^2}\qpoch{-q^2}{q^2}
-\qpoch{-q}{q^2}\qpoch{q^2}{q^2}\Bigr) \nonumber \\
&=\frac14\Bigl(\qpoch{-q}{q^2}+\qpoch{q}{q^2}\Bigr) \\
&\qquad\times \Bigl(\qpoch{-q^2}{q^2}-\qpoch{q^2}{q^2}\Bigr). \label{A-help1}
\end{align}

Now apply Jacobi's triple product~\eqref{Jacobi-triple} with $q\to q^2$ and with $z=q^{-1}$ and $z=-q^{-1}$, respectively. This yields
\[
\sum_{n\in\mathbb Z} q^{2n^2-n}
=(q^4;q^4)_\infty \qpoch{-q}{q^2},
\]
and
\[
\sum_{n\in\mathbb Z} (-1)^n q^{2n^2-n}
=(q^4;q^4)_\infty \qpoch{q}{q^2}.
\]
Therefore
\begin{align}
\frac12\Bigl(\qpoch{-q}{q^2}+\qpoch{q}{q^2}\Bigr)
&=\frac{1}{\qpoch{q^4}{q^4}}
\sum_{n\in\mathbb Z}\frac{1+(-1)^n}{2}\,q^{2n^2-n} \nonumber \\
&=\frac{1}{\qpoch{q^4}{q^4}}
\sum_{m\in\mathbb Z} q^{8m^2-2m}.\label{A-help2}
\end{align}

In addition, applying the Jacobi triple product~\eqref{Jacobi-triple} with $q\to q^2$ and $z=-1$, and then using Euler's identity~\eqref{Euler0}, we obtain
\[
\sum_{n\in\mathbb Z}(-1)^n q^{2n^2}
=(q^2,q^2,q^4;q^4)_\infty
=(q^2;q^2)_\infty (q^2;q^4)_\infty
=1+2\sum_{n\ge 1}(-1)^n q^{2n^2}.
\]
Hence
\begin{align}
\frac12\Bigl(\qpoch{-q^2}{q^2}-\qpoch{q^2}{q^2}\Bigr)
&=\frac{1}{2}(-q^2;q^2)_\infty \Big( 1-\frac{(q^2;q^2)_\infty}{(-q^2;q^2)_\infty} \Big) \nonumber \\
&=\qpoch{-q^2}{q^2}\sum_{m\ge 1}(-1)^{m-1}q^{2m^2}.\label{A-help3}
\end{align}

Substituting~\eqref{A-help2} and~\eqref{A-help3} in~\eqref{A-help1} yields
\begin{align*}
A(q^2)
&=\frac{\qpoch{-q^2}{q^2}}{\qpoch{q^4}{q^4}}
\left(\sum_{m\in\mathbb Z} q^{8m^2-2m}\right)
\left(\sum_{m\ge 1}(-1)^{m-1}q^{2m^2}\right) \\
&=\frac{1}{\qpoch{q^2}{q^2}}
\left(\sum_{m\in\mathbb Z} q^{8m^2-2m}\right)
\left(\sum_{m\ge 1}(-1)^{m-1}q^{2m^2}\right).
\end{align*}
Finally, replacing $q^2$ by $q$ gives
\[
A(q)
=\frac{1}{\qpoch{q}{q}}
\left(\sum_{m\in\mathbb Z} q^{4m^2-m}\right)
\left(\sum_{m\ge 1}(-1)^{m-1}q^{m^2}\right).
\]
This completes the proof of~\eqref{A344650-A}.
\end{proof}
\section{Concluding remarks and open problems}\label{sec:remarks}
We conclude with two combinatorial problems and two conjectures.

\begin{problem}
Give a direct combinatorial proof of the identity $a_e(N)=b_e(N)$ in Theorem~\ref{thm:Te-partitions}; equivalently, construct a bijection between the ordinary partitions counted by $a_e(N)$ and the overpartitions counted by $b_e(N)$.
\end{problem}

\begin{problem}
Find a combinatorial proof of Theorem~\ref{thm:mod4}.
\end{problem}

\begin{remark}
The coefficients of $C(q)$ are much sparser than those of $T_o(q)$, but they do not remain in $\{-2,-1,0,1,2\}$. For example,
\[
c(9884)=3,\qquad c(11956)=-3,\qquad c(94836)=4,\qquad c(310324)=-5.
\]
\end{remark}

\begin{conjecture}
\begin{align}
\limsup_{n\to\infty}c(n)&=+\infty, \label{limsup} \\
\liminf_{n\to\infty}c(n)& =-\infty. \label{liminf}
\end{align}
\end{conjecture}

Write
\[
T_o(q)=\sum_{n\ge 0} t_o(n)q^n.
\]
By Proposition~\ref{prop:companions},
\[
4t_o(n)=s_1(n+1)-\varepsilon(n+1),
\]
where
\[
\varepsilon(m)=
\begin{cases}
(-1)^r,& m=\trinum{r}\text{ for some }r\ge 0,\\
0,& \text{otherwise.}
\end{cases}
\]
Thus, whenever $n+1$ is not triangular,
\[
t_o(n)=\frac{s_1(n+1)}{4}.
\]
In particular, any $2$-adic divisibility statement for $s_1(m)$ on nontriangular indices transfers directly to $t_o(m-1)$.

\begin{conjecture}\label{conj:T-parity}
If $8n+9$ has a prime divisor $p\equiv 5,7\pmod{8}$ to odd exponent, equivalently if $8n+9$ is not represented by $x^2+2y^2$, then
\[
t_o(n)\equiv 0\pmod{2}.
\]
\end{conjecture}

\noindent{\bf Data Availability Statement.\ }
Not applicable.

\bigskip

\noindent{\bf Declarations.\ }
The authors declare that they have no conflict of interest.

\section*{Acknowledgment}
We thank the referee for valuable comments and suggestions that improved the presentation of this paper.

\end{document}